\documentclass[11pt]{amsart}
\usepackage{amssymb,amsmath}
\usepackage{caption}
\usepackage{verbatim}
\usepackage{epsfig}

\def\Bbb{\mathbb}
\def\Cal{\mathcal}
\def\Dt{\partial_t}

\def\eb{\varepsilon}

\def\R {\mathbb{R}}

\def\<{\left<}
\def\>{\right>}

\def\Nx{\nabla_x}
\def\Dx{\Delta_x}
\def\({\left(}
\def\){\right)}

\def\divv{\operatorname{div}}
\def\rot{\operatorname{curl}}

\def\R{\Bbb R}
\def\Dx{\Delta_x}
\def\Nx{\nabla_x}
\def\Dt{\partial_t}
\def\divv{\operatorname{div}}
\def\thetag{\eqref}

\def\R{\Bbb R}
\def\eb{\varepsilon}
\def\({\left(}
\def\){\right)}
\def\Dt{\partial_t}
\def\Dx{\Delta_x}
\def\Nx{\nabla_x}
\def\divv{\operatorname{div}}
\def\rott{\operatorname{curl}}

\newtheorem{proposition}{Proposition}[section]
\newtheorem{theorem}[proposition]{Theorem}
\newtheorem{corollary}[proposition]{Corollary}
\newtheorem{lemma}[proposition]{Lemma}

\theoremstyle{definition}
\newtheorem{definition}[proposition]{Definition}

\newtheorem{remark}[proposition]{Remark}

\numberwithin{equation}{section}

\def \no#1#2#3 {{\bf #1} (#3), #2.}
\def \eds#1#2#3 {#1, #2, #3.}
\begin{document}
\title[Dissipative Euler Equations] {Infinite energy solutions for Dissipative Euler equations in $\R^2$}

\author[V. Chepyzhov and S. Zelik] {Vladimir Chepyzhov${}^1$ and Sergey Zelik${}^2$}

\subjclass[2000]{35Q30,35Q35}
\keywords{Euler equations, Ekman damping, infinite energy solutions, weighted energy estimates, unbounded domains}
\thanks{
The second author would like to thank  Thierry Gallay  for the fruitful
discussions.}
\address{${}^1$ Institute for Information Transmission Problems, RAS,\newline
Bolshoy Karetniy 19, Moscow 101447, Russia;\newline
National Research University Higher School of Economics,\newline
Myasnitskaya Street 20, Moscow 101000, Russia},
\email{chep@iitp.ru}

\address{${}^2$ University of Surrey, Department of Mathematics, \newline
Guildford, GU2 7XH, United Kingdom.}
\email{s.zelik@surrey.ac.uk}

\begin{abstract} We study the  Euler equations with the so-called Ekman damping in the whole 2D space. The global well-posedness and dissipativity for the  weak infinite energy solutions of this problem in the uniformly local spaces is verified based on the further development of the weighted energy theory for the Navier-Stokes and Euler type problems. In addition, the existence of weak locally compact global attractor is proved and some extra compactness of this attractor is obtained.
\end{abstract}
\thanks{This work is partially supported by the Russian Foundation of Basic Researches (projects 14-01-00346 and 15-01-03587) and by  the grant  14-41-00044 of RSF}

 \maketitle
\tableofcontents
\section{Introduction}\label{s0}

We study the following dissipative Euler system in the whole plane $x\in\R^2$:
\begin{equation}\label{1}
\begin{cases}
\Dt u+(u,\Nx)u+\alpha u+\Nx p=g,\\
\divv u=0,\ \ u\big|_{t=0}=u_0,
\end{cases}
\end{equation}
which differs from the classical Euler equations by the presence of the so-called   Ekman damping term $\alpha u$ with $\alpha>0$.
These equations describe, for instance, a 2-dimensional fluid moving on a rough surface and are used in
geophysical models for large-scale processes in atmosphere and ocean. The term $\alpha u$ parameterizes
the main dissipation occurring in the planetary boundary layer (see, e.g., \cite{Ped79}; see also \cite{BP06} for the alternative source of  damped Euler equations).
\par
The mathematical features of these and related equations are studied in a number of papers (see, for instance, \cite{BCT,BF,CV08,Il91,IMT04,IlT06}) including the analytic properties (which are very similar to the classical Euler equations without dissipative term, see \cite{BT,Li,Li1,Tem77} and references therein), stability analysis, vanishing viscosity limit, etc.
\par
The attractors for damped Euler equations \eqref{1} in the case of bounded underlying domains have been studied in \cite{Il91,BF, IlT06,IMT04,CV08,CVZ11}. Remind that, in contrast to the Navier-Stokes case, the damped Euler equations remain hyperbolic and we do not have any smoothing property on a finite time interval. Moreover, even the asymptotic smoothness as time tends to infinity is much more delicate here. Indeed, similar to the classical Euler equations, following to Yudovich, see \cite{Yu1}, we have the global existence of smooth solutions, but  the best possible estimate for the smooth norms of these solutions grow faster than exponential in time, so they are not helpful for the attractor theory. Thus, to the best of our knowledge, there is no way to obtain  more regularity of the attractor than the $L^\infty(\Omega)$ bounds for the vorticity even in the case of bounded underlying domain $\Omega$ or periodic boundary conditions. On the other hand, the $L^\infty$-bounds for the vorticity cannot be essentially relaxed if we want to have the uniqueness of a solution. By this reason, the  {\it weak} attractors are normally used in order to describe the longtime behavior of solutions of the damped Euler equations. Some exception is the paper \cite{CVZ11} where the so-called trajectory attractor is constructed for this system in a {\it strong} topology of $W^{1,2}(\Omega)$ based on the enstrophy equality and the energy method.
\par
The situation becomes more complicated where the underlying domain becomes unbounded, say, $\Omega=\R^2$ and we are interested in the infinite energy solutions. Indeed, although in this case we have an immediate control of the $L^\infty$-norm of the vorticity from the maximum principle:
\begin{equation}
\|\omega(t)\|_{L^\infty}\le Ce^{-\alpha t}\|\omega(0)\|_{L^\infty}+\frac1\alpha\|\rott g\|_{L^\infty},\  \omega:=\rott u,
\end{equation}
this gives only growing in time (faster than exponentially) estimates for the velocity $u$, see \cite{GMS01,ST07} even in a more simple case of damped Navier-Stokes equations (see also \cite{Ser95,Kell} for the analogous results for Euler equations), so in order to get the dissipative bounds for the velocity field, we need to use the energy type estimates. For the case of damped Navier-Stokes equations these estimates have been obtained in \cite{Zel-JMFM} for the case where the initial data $u_0$ belong to the so-called uniformly local Sobolev spaces, see also \cite{Z07,Z08} for the analogous results for the case of Navier-Stokes equations in cylindrical domains as well as 
\cite{EMZ04,EZ01,MZ08} 
and references therein for general theory of dissipative PDEs in unbounded domains.
\par
The aim of the present paper is to build up the attractor theory for the damped Euler equation \eqref{1} in uniformly local spaces extending the results of \cite{CVZ11} and \cite{Zel-JMFM}.
 We assume that $g\in \Cal H_b$, where
\begin{equation}\label{2}
\Cal H_b:=\bigg\{u\in[L^2_b(\R^2)]^2,\ \divv u=0,\ \rott u\in L^\infty(\R^2)\bigg\}
\end{equation}
and $L^2_b(\R^2)$ is the usual uniformly local space determined by the norm
\begin{equation}\label{3}
\|u\|_{L^2_b}:=\sup_{x_0\in\R^2}\|u\|_{L^2(B^1_{x_0})}<\infty.
\end{equation}
Here and below, $B^R_{x_0}$ means the unit ball of radius $R$ in $\R^2$ centered at $x_0\in\R^2$. The norm in the space $\mathcal H_b$ is defined by the following natural formula:
\begin{equation}\label{0.h-norm}
\|u\|_{\mathcal H_b}:=\|u\|_{L^2_b}+\|\rott u\|_{L^\infty},
\end{equation}
 see Section \ref{s1} for more details.
\par
By definition, function $u=u(t,x)$ is a weak solution of \thetag{1} if $u(t)\in\Cal H_b$ for $t\ge0$
and satisfies equation \thetag{1} in the sense of distributions, see Definition \ref{Def0.1} below.
\par
The main result of the paper is the following theorem.

\begin{theorem}\label{Th0.main} Let the external forces $g\in\Cal H_b$. Then, for every $u_0\in\Cal H_b$, problem \eqref{1} possesses a unique weak solution $u(t)$ and this solution satisfies the following dissipative estimate:
\begin{equation}\label{0.dis}
\|u(t)\|_{\Cal H_b}\le Q(\|u_0\|_{\Cal H_b})e^{-\beta t}+Q(\|g\|_{\Cal H_b}),
\end{equation}
where the positive constant $\beta$ and monotone increasing function $Q$ are independent of $t\ge0$ and $u_0\in\Cal H_b$. The solution semigroup $S(t):\Cal H_b\to\Cal H_b$ associated with this equation possesses a weak locally compact global attractor $\Cal A$ in the phase space $\Cal H_b$, see Definition \ref{Def5.attr}. Moreover, this attractor is a compact set in $[W^{1,p}_{loc}(\R^2)]^2$, for every $p<\infty$ and attracts bounded sets of $\Cal H_b$ in the topology of $[W^{1,p}_{loc}(\R^2)]^2$.
\end{theorem}

The paper is organized as follows.
\par
In Section \ref{s1}, we recall the definitions and basic properties of the weighted and uniformly local Sobolev spaces, introduce special classes of weights and remind a number of elementary inequalities which will be used throughout of the paper.
\par
In Section \ref{s2}, we introduce (following \cite{Zel-JMFM}) a number of technical tools which allows us to treat the pressure term in the proper weighted and uniformly local spaces as well as to exclude it from the various weighted energy estimates.
\par
Section \ref{s3} is devoted to the derivation of the basic dissipative estimate \eqref{0.dis}. In this section, based on the new version of the interpolation inequality (proved in the Appendix), we extend the method initially suggested in \cite{Zel-JMFM} for the case of damped Navier-Stokes equations to more difficult case of zero viscosity. Moreover, we also indicate here some improvements of the results concerning the classical Navier-Stokes and Euler equations (which corresponds to the case of $\alpha=0$). In this case, the solution $u(t)$ may grow as $t\to\infty$ and as proved in \cite{Zel-JMFM} the growth rate is restricted by the quintic polynomial in time. As indicated in Remark \ref{Rem3.NS0}, using the approach developed in this paper, we may replace the quintic polynomial by the cubic one:
\begin{equation}
\|u(t)\|_{L^2_b(\R^2)}\le C(t+1)^3.
\end{equation}
Moreover, in the particular case $g=0$ this estimate can be further improved:
\begin{equation}
\|u(t)\|_{L^2_b(\R^2)}\le C(t+1),
\end{equation}
see also the recent work \cite{Gal14} where the analogous linear growth estimate has been established for the infinite energy solutions of the Navier-Stokes equations.
\par
The uniqueness of the weak solution of \eqref{1} is verified in Section \ref{s4} by adapting the famous Yudovich proof to the case of weighted and uniformly local spaces. Note that, in contrast to \cite{Ser95} and \cite{Kell}, our approach does not use the so-called Serfati identity and is based on the Yudovich type estimates in weighted $L^2$-spaces.
 Moreover, following \cite{DiL}, we establish here the so-called weighted enstrophy equality which plays a crucial role in verifying the strong compactness of the attractor.
\par
Finally, the weak locally compact attractor $\Cal A$ is constructed in Section \ref{s5}. Moreover, using the above mentioned weighted enstrophy equality and the energy method (analogously to \cite{CVZ11}), we prove the compactness of this weak attractor in the strong topology of the space $[W^{1,p}_{loc}(\R^2)]^2$ for any $p<\infty$.
\par
Note also that, in contrast to the case of Navier-Stokes equations, in the case of Euler equations, the $L^\infty$-estimate for the
vorticity holds not only for $\R^2$, but for more or less general unbounded domains. This allows to extend the results of the paper to the case of unbounded domains different from $\R^2$. To this end, one just needs to modify formula \eqref{0.6} for pressure by including the proper boundary terms. We return to this problem somewhere else.

\section{Preliminaries I: Weighted and uniformly local spaces}\label{s1}

In this section, we briefly discuss the definitions and basic properties of the weighted and uniformly local Sobolev spaces (see \cite{MZ08,Z07,Z03} for more detailed exposition). We start with the class of admissible weight functions and associated weighted spaces.

\begin{definition}\label{Def0.5.weights} A positive function $\phi(x)$, $x\in\R^2$, is a weight function of exponential growth rate $\mu\ge0$ if
\begin{equation}\label{1.weight}
\phi(x+y)\le C e^{\mu|y|}\phi(x),\ \ x,y\in\R^2.
\end{equation}
The associated weighted Lebesgue space $L^p_\phi(\R^2)$, $1\le p<\infty$, is defined as a subspace of functions belonging to $L^p_{loc}(\R^2)$ for which the following norm is finite:
\begin{equation}\label{1.wspace}
\|u\|^p_{L^p_{\phi}}:=\int_{\R^2}\phi(x)|u(x)|^p\,dx<\infty
\end{equation}
and the Sobolev space $W^{l,p}_\phi(\R^2)$ is the subspace of distributions $u\in \mathcal D'(\R^2)$ whose derivatives up to order $l$ inclusively belong to $L^p_\phi(\R^2)$ (this works for positive integer $l$ only, for fractional and negative $l$, the space $W^{l,p}_\phi$ is defined using the interpolation and duality arguments, see \cite{EZ01,Z07} for more details).
\end{definition}
The typical examples of weight functions of exponential growth rate are
\begin{equation}\label{1.exweight}
\phi(x):=e^{-\eb|x-x_0|}\ \ \text{or}\ \ \phi(x):=e^{-\sqrt{1+\eb^2|x-x_0|^2}},\ \ \eb\in\R,\ \ x_0\in\R^2.
\end{equation}
Another class of admissible weights of exponential growth rate are the so-called polynomial weights and, in particular, the weight function
\begin{equation}\label{1.theta}
\theta_{x_0}(x):=\frac1{1+|x-x_0|^3},\ \ x_0\in\R^2,
\end{equation}
which will be essentially used throughout of the paper.
\par
Next, we define the so-called uniformly local Sobolev spaces.

\begin{definition}\label{Def1.ul} The space $L^p_b(\R^2)$ is defined as the subspace of functions of $L^p_{loc}(\R^2)$ for which the following norm is finite:
\begin{equation}\label{1.ul}
\|u\|_{L^p_b}:=\sup_{x_0\in\R^2}\|u\|_{L^p(B^1_{x_0})}<\infty
\end{equation}
(here and below $B^R_{x_0}$ denotes the $R$-ball in $\R^2$ centered at $x_0$).
The spaces $W^{l,p}_b(\R^2)$  are defined as subspaces of distributions $u\in\mathcal D'(\R^2)$ whose derivatives up to order $l$ inclusively belong to the space $L^p_b(\R^2)$.
\par
\end{definition}

The next  proposition gives the useful equivalent norms in the weighted Sobolev spaces
\begin{proposition}\label{Prop1.equiv} Let $\phi$ be the weight function of exponential growth rate and let $1\le p<\infty$, $l\in\mathbb R$ and $R>0$.
 Then,
\begin{equation}\label{1.equiv}
C_1\int_{x_0\in\R^2}\phi(x_0)\|u\|_{W^{l,p}(B^R_{x_0})}^p\,dx_0\le \|u\|_{W^{l,p}_\phi}^p\le C_2\int_{x_0\in\R^2}\phi(x_0)\|u\|_{W^{l,p}(B^R_{x_0})}^p\,dx_0,
\end{equation}
where the constants $C_i$ depend on $R$, $l$ and $p$ and the constants $C$ and $\mu$ from \eqref{1.weight}, but are independent of $u$ and of the concrete choice of the weight $\phi$.
\end{proposition}
For the proof of these estimates, see e.g., \cite{EZ01}.
\par
Thus, the norms $\int_{x_0\in\R^2}\phi(x_0)\|u\|_{W^{l,p}(B^R_{x_0})}^p\,dx_0$ computed with different $R$'s are equivalent.
\par
The next Proposition gives relations between the weighted and uniformly local norms.

\begin{proposition}\label{Prop1.more} Let $\phi$ be the weight  of exponential growth rate such that $\int_{x\in\R^2}\phi\,dx<\infty$. Then, for every $u\in W^{l,p}_b(\R^2)$ and every $\kappa\ge 1$,
\begin{equation}\label{1.ul-w}
\|u\|_{W^{l,p}(B^\kappa_{x_0})}^p\le C\int_{y\in B^\kappa_{x_0}}\|u\|_{W^{l,p}(B^1_y)}^p\,dy\le C_\kappa\int_{y\in\R^2}\phi(y-x_0)\|u\|^p_{W^{l,p}(B^1_{y})}\,dy
\end{equation}
and, in particular, fixing $\kappa=1$ in \eqref{1.ul-w} and taking the supremum with respect to $x_0\in\R^2$, we have
\begin{equation}\label{1.ul-w1}
\|u\|_{W^{l,p}_b}\le C\sup_{x_0\in\R^2}\|u\|_{W^{l,p}_{\phi(\cdot-x_0)}},
\end{equation}
where $C$ is independent of $u$ and the concrete choice of the weight $\phi$. In addition,
\begin{equation}\label{1.w-ul}
\|u\|_{W^{l,p}_\phi}^p\le C\|\phi\|_{L^1}\|u\|_{W^{l,p}_b}^p,
\end{equation}
where $C$ is also independent of $u$ and the concrete choice of $\phi$.
\end{proposition}
For the proof of these results, see e.g., \cite{Z07,Z03}.
\par
The next lemma gives a simple, but important estimate for the weights $\theta_{x_0}(x)$ which will allow us to handle the convolution operators in weighted spaces.
\begin{lemma}\label{Lem1.weight} Let $\theta_{x_0}(x)$ be the weight defined via \eqref{1.theta}. Then, the following estimate holds:
\begin{equation}\label{1.thetakey}
\int_{x\in\R^2}\theta_{x_0}(x)\theta_{y_0}(x)\,dx\le C\theta_{x_0}(y_0),
\end{equation}
where $C$ is independent of $x_0,y_0\in\R^2$.
\end{lemma}
For the proof of this lemma see, e.g., \cite{Zel-JMFM}.
\begin{corollary}\label{Cor1.conv} Let $\theta_{x_0}(x)$ be defined via \eqref{1.theta}. Then, for every $u\in L^p_{\theta_{x_0}}(\R^2)$, we have
\begin{equation}\label{1.inv}
\|u\|_{L^p_{\theta_{y_0}}}^p\le C\int_{x_0\in\R^2}\theta_{y_0}(x_0)\|u\|^p_{L^p_{\theta_{x_0}}}\,dx_0,
\end{equation}
where $C$ is independent of $y_0\in\R$.
\end{corollary}
The proof of this corollary can also be found in \cite{Zel-JMFM}.

We conclude this section by introducing some weights and norms depending on a big parameter $R$ which will be crucial for what follows. First, we introduce the following equivalent norm in the space $W^{l,p}_b(\R^2)$:
\begin{equation}
\|u\|_{W^{l,p}_{b,R}}:=\sup_{x_0\in\R^2}\|u\|_{W^{l,p}(B^R_{x_0})}.
\end{equation}
Then, according to \eqref{1.ul-w},
\begin{equation}\label{1.ulR}
\|u\|_{W^{l,p}_b}\le\|u\|_{W^{l,p}_{b,R}}\le CR^{2/p}\|u\|_{W^{l,p}_b},
\end{equation}
where the constant $C$ is independent of $R\ge1$. We also introduce the scaled weight function
\begin{equation}\label{ttheta}
\theta_{R,x_0}(x):=\frac1{R^3+|x-x_0|^3}=R^{-3}\theta_{x_0/R}(x/R).
\end{equation}
Then, the scaled analogue of \eqref{1.ul-w} reads
\begin{equation}\label{1.ul-wR}
\|u\|_{W^{l,p}(B^{\kappa R}_{x_0})}^p\le CR^{-2}\int_{y\in B^{\kappa R}_{x_0}}\|u\|_{W^{l,p}(B^R_y)}^p\,dy\le C_\kappa R\int_{y\in\R^2}\theta_{R,x_0}(y)\|u\|^p_{W^{l,p}(B^R_{y})}\,dy,
\end{equation}
where the constants $C$ and $C_\kappa$ are independent of $R$ and the scaled analogue of \eqref{1.thetakey}
\begin{equation}\label{1.R-thetakey}
\int_{x\in\R^2}\theta_{R,x_0}(x)\theta_{R,y_0}(x)\,dx\le CR^{-1}\theta_{R,x_0}(y_0),
\end{equation}
where $C$ is independent of $x_0,y_0\in\R^2$ and $R>0$. Moreover,  multiplying  inequality \eqref{1.ul-wR} by $\theta_{R,y_0}(x_0)$, integrating over $x_0$ and using \eqref{1.R-thetakey}, we see that
\begin{equation}\label{1.eqtheta}
\int_{x\in\R^2}\theta_{R,x_0}(x)\|u\|^p_{W^{l,p}(B^{\kappa R}_x)}\,dx\le C_\kappa \int_{x\in\R^2}\theta_{R,x_0}(x)\|u\|^p_{W^{l,p}(B^{R}_x)}\,dx,
\end{equation}
where $C_\kappa$ is independent of $R$. We also note that, analogously to \eqref{1.w-ul} and using \eqref{1.ulR},
\begin{equation}\label{ulR}
\int_{x\in\R^2}\theta_{R,x_0}(x)\|u\|^2_{W^{l,2}(B^R_{x})}\,dx\le CR^{-1}\|u\|^2_{W^{l,2}_{b,R}}\le C_1R\|u\|^2_{W^{l,2}_b},
\end{equation}
where the constants $C$ and $C_1$ are independent of $R\gg1$.

\section{Preliminaries II: Estimating the pressure} \label{s2}
In this section, we introduce the key estimates which allow us to work with the pressure term $\nabla p$ in the uniformly local spaces. Note that the Helmholtz decomposition does not work for the general vector fields belonging to $L^2_b(\R^2)$, so the standard (for the bounded domains) approach
does not work at least directly and we need to proceed in a bit more accurate way.
\par
As usual, we assume that \eqref{1} is satisfied in the sense of distributions. Then, taking  the divergence from both sides of \eqref{1} and assuming that the external forces $g$ are divergence free:
\begin{equation}\label{gdiv}
\divv g=0,
\end{equation}
we have
\begin{equation}\label{0.5}
-\Dx p=\divv((u,\Nx)u)=\sum_{i,j=1}^2\partial_{x_i}\partial_{x_j}(u_iu_j).
\end{equation}
Thus, formally, $p$ can be expressed through $u$ by the following singular integral operator:
\begin{equation}\label{0.6}
p(y):=\int_{\R^2}\sum_{ij}K_{ij}(x-y)u_i(x)u_j(x)\,dx,\ \ K_{ij}(x):=\frac1{2\pi}\frac{|x|^2\delta_{ij}-2x_ix_j}{|x|^4}
\end{equation}
which we present in the form
\begin{equation}\label{0.conv}
p=\Bbb K w:=K*w,\ \ w:=u\otimes u,\ \ K*w=\sum_{ij} K_{ij} * w_{ij}.
\end{equation}
It is well-known that the convolution operator $\Bbb K$ is well-defined as a bounded  linear operator from $w\in [L^{q}(\R^2)]^4$ to $p\in L^q(\R^2)$,
$1<q<\infty$, but it is not true neither for $q=\infty$ nor for the uniformly local space $L^q_b(\R^2)$. However, as the following lemma shows, the {\it gradient} of $p$ (which is sufficient in order to define a solution of \eqref{1}) is
well-defined in uniformly local spaces and has natural regularity properties.

\begin{lemma}\label{Lem0.1} The operator $w\to\Nx p$, where $p$ is defined via \eqref{0.conv} can be extended by continuity (in $L^q_{loc}$) in a unique way to the bounded operator from $[W^{s,q}_b(\R^2)]^4$ to $[W^{s-1,q}_b(\R^2)]^2$, $1<q<\infty$ and $s\in\R$.
\end{lemma}
For the proof of this lemma see \cite{Zel-JMFM}.
We will denote the operator obtained in the lemma by $\Nx P$ and the corresponding term in the Euler equation will be denoted by $\Nx P(u\otimes u)$. Then, in particular
\begin{equation}\label{0.9}
\Nx P(u\otimes u): [W^{1,2q}_b(\R^2)]^2\to [L^{q}_b(\R^2)]^2,\ 1<q<\infty.
\end{equation}
Here the operator $\Nx P(u\otimes u)$ is considered as a nonlinear (quadratic) operator $u\to\Nx P(u\otimes u)$. \par
We are now ready to give the definition of a weak solution of problem \eqref{1}.
\begin{definition}\label{Def0.1} Let the external forces $g\in\Cal H_b$, where the space $\Cal H_b$ is defined as follows:
\begin{equation}\label{2.space}
\Cal H_b:=\{u\in[L^2_b(\R^2)]^2,\ \divv u=0,\ \rott u\in L^\infty(\R^2)\}.
\end{equation}
Here and below $\rott u:=\partial_{x_2}u_1-\partial_{x_1}u_2$ and the norm in this space is given by \eqref{0.h-norm}.
\par
 A  vector field $u(t,x)$ is a weak solution of the damped Euler problem \eqref{1} if
\begin{equation}\label{2.defsol}
u\in L^\infty(\R_+,\Cal H_b)
\end{equation}
and the equation is satisfied in the sense of distributions with $\Nx p=\Nx P(u\otimes u)$ defined in Lemma \ref{Lem0.1}. Note that, according to the interpolation and embedding theorems,
\begin{equation}\label{2.sm}
u\in L^\infty(\R_+,W^{1,q}_b(\R^2))\subset L^\infty(\R_+\times\R^2),\ \ 2<q<\infty,
\end{equation}
see e.g., \cite{L02}. Thus, $u\otimes u\in  L^\infty(\R_+,W^{1,q}_b(\R^2))$, so, due to the previous lemma, the pressure term $\Nx p\in L^\infty([0,T],L^q_b(\R^2))$ and the equation \eqref{1} can be understood as equality in this space. Moreover, from equation \eqref{1}, we then conclude that $\Dt u\in L^\infty(\R_+,L^q_b(\R^2))$ and, therefore, $u\in C([0,T],L^q_b(\Omega))$. Thus, the initial data is also well-defined.
\end{definition}
\begin{remark} We emphasize once more that only the gradient of pressure $\nabla p$ is well-defined as an element of $L^\infty([0,T],L^q_b)$, but the pressure itself may be unbounded as $|x|\to\infty$. To be more precise, the operator $\nabla P$ defined above satisfies
\begin{equation}\label{1.press}
-\divv \nabla P(w)=\sum_{ij}\partial_{x_i}\partial_{x_j} w_{ij},\ \ \rot \nabla P(w)=0
\end{equation}
in the sense of distributions. These relations can be justified by approximating $w$ by finite
functions and passing to the limit analogously to the proof of Lemma \ref{Lem0.1} given in \cite{Zel-JMFM}. Therefore, there is a function $p\in L^\infty([0,T],W^{1,q}_{loc}(\R^2))$ such that
$$
\nabla p=\nabla P(w),
$$
see \cite{Tem77}, but this function may grow as $|x|\to\infty$ (in a fact, one can only guarantee that $p\in BMO(\R^2)$, but the functions with bounded mean oscillation may grow as $|x|\to\infty$, say, as a polynomial of $\log|x|$, see \cite{L02}). Thus, in general $p\notin L^\infty(\R^2)$.
\par
 Note also that the choice of $\nabla p=\nabla P(u\otimes u)$ is not unique. However, if $p_1$ and $p_2$ both satisfy \eqref{0.5} (for the same velocity field $u$), then the difference $p_1-p_2$ solves $\Delta(p_1-p_2)=0$ (in the sense of distributions) and, consequently is a {\it harmonic} function. Moreover, every harmonic function with bounded gradient is linear, so, if we want the velocity field $u$ to be in the proper uniformly local space, the most general choice of the pressure is
\begin{equation}
\nabla p=\nabla P(u\otimes u)+\vec C(t),
\end{equation}
where the constant vector $\vec C(t)$ depends only on time (and is independent of $x$) and $\nabla P$ is defined in Lemma \ref{Lem0.1}. In the present paper, we consider only the choice $\vec C(t)\equiv0$. In a fact, the vector $\vec C(t)$ should be treated as one more external data and can be chosen arbitrarily, but this does not lead to more general theory since everything can be reduced to the case of $\vec C\equiv 0$ by replacing the external force $g$ by $g-\vec C(t)$.
\end{remark}

We conclude this preliminary section by reminding the key estimate which allows us handle the pressure term in weighted energy estimates, see \cite{Zel-JMFM} for more details.
To this end, we introduce for every $x_0\in\R^2$ and $R>1$ the
cut-off function $\varphi_{R,x_0}$ which satisfies
\begin{equation}\label{0.11}
\varphi_{R,x_0}(x)\equiv1, \text{ for } x\in B^R_{x_0},\ \ \ \varphi_{R,x_0}(x)\equiv 0, \text{ for } x\notin B^{2R}_{x_0},
\end{equation}
and
\begin{equation}\label{0.12}
|\Nx\varphi_{R,x_0}(x)|\le CR^{-1}\varphi^{1/2}_{R,x_0}(x),
\end{equation}
where $C$ is independent of $R$ (obviously such family of cut-off functions exist). Then, the following result holds.

\begin{lemma}\label{Lem0.2} Let the exponents $1<p,q<\infty$, $\frac1p+\frac1q=1$, $w\in [L^p_b(\R^2)]^4$ and $v\in [W^{1,q}(\R^2)]^2$ be
divergence free. Then the following estimate holds:
\begin{equation}\label{0.15}
|(\Nx P(w),\varphi_{R,x_0}v)|\le C\int_{\R^2}\theta_{R,x_0}(x)\|w\|_{L^p(B^R_{x})}\,dx\cdot\|\varphi_{R,x_0}^{1/2}v\|_{L^q},
\end{equation}
where $C$ is independent of $R$ and $x_0$ and $\theta_{R,x_0}(x)$ is defined by \eqref{ttheta}.
\end{lemma}
For the proof of this lemma, see \cite{Zel-JMFM}.

\section{Dissipative estimates for the velocity field}\label{s3}
 Our aim here is to prove the following dissipative estimate for the solutions of \thetag{1} in the phase space $\Cal H_b$.
 We start with recalling the $L^\infty$-estimate for the vorticity $\omega=\rott u$ which satisfies the following scalar transport equation:
 \begin{equation}\label{3.vort}
 \Dt \omega+\alpha\omega+(u,\Nx)\omega=\rott g,\ \ \omega\big|_{t=0}=\omega_0:=\rott u_0.
 \end{equation}
 \begin{lemma}\label{Lem3.vort} Let $u$ be a weak solution of the Euler problem \thetag{1}. Then, the vorticity $\omega$ satisfies the following estimate:
\begin{equation}\label{9}
\|\omega(t)\|_{L^\infty}\le C\|w(0)\|_{L^\infty}e^{-\alpha t}+\frac{\|\rott g\|_{L^\infty}}\alpha,
\end{equation}
 where the constant $C$ is independent of $t$ and $u$.
 \end{lemma}
Indeed, the desired estimate \eqref{9} is an immediate corollary of the maximum principle applied to the transport equation \eqref{3.vort}. The validity of the maximum principle can be easily justified using the fact that the weak solution of the damped Euler equation is unique (which will be proved in the next section) and approximating the solution $u$ by the smooth ones by the vanishing viscosity method. Thus, we only need to estimate the $L^2_b$-norm of $u$. To do that, we will extend the approach developed in \cite{Zel-JMFM} to the case of Euler equations.

\begin{theorem}\label{Th3.main} Let the above assumptions hold. Then the Euler equation \eqref{1} possesses at least one weak solution which satisfies the following dissipative estimate:
\begin{equation}\label{8}
\|u(t)\|_{\Cal H_b}\le Q(\|u_0\|_{\Cal H_b})e^{-\beta t}+Q(\|g\|_{\Cal H_b}),\ \
\end{equation}
where $\beta>0$ and $Q$ is a monotone function.
\end{theorem}
\begin{proof} For simplicity, we first derive the desired estimate \eqref{8} in the non-dissipative case $\beta=0$ and then indicate the changes to be made in order to verify the dissipation. The existence of a solution can be obtained after that in a standard way using e.g., the vanishing viscosity method which we leave to the reader.
 \par
 We will systematically use the weight functions
\begin{equation}\label{10}
\theta_{R,x_0}(x):=\frac1{R^3+|x-x_0|^3}
\end{equation}
and the family of cut-off functions $\varphi_{R,x_0}(x)$ which equal to one if $x\in B^R_{x_0}$ and zero outside of $B^{2R}_{x_0}$ such that
\begin{equation}\label{11}
|\Nx \varphi_{R,x_0}(x)|\le CR^{-1}\varphi_{R,x_0}(x)^{1/2}.
\end{equation}
introduced in Sections \ref{s1} and \ref{s2}.
\par
Then, multiplying equation \thetag{1} by $u\varphi_{R,x_0}$ (where $R$ is sufficiently large number and $x_0\in\R^2$) and following \cite{Zel-JMFM}, after the integration over $x$, estimation of the nonlinear term as follows:
$$
|((u,\Nx u),\varphi_{R,x_0}u)|=\frac12|(u.\Nx\varphi_{R,x_0},|u|^2)|\le CR^{-1}\|u\|_{L^3(B^{2R}_{x_0})}^3
$$
and standard transformations, we get
\begin{equation}\label{Z.13}
\frac d{dt}\|u(t)\|_{L^2_{\varphi_{R,x_0}}}^2+\alpha\|u\|^2_{L^2_{\varphi_{R,x_0}}}\le C\|g\|_{L^2_{\varphi_{R,x_0}}}^2+
 CR^{-1}\|u\|^3_{L^3(B^{2R}_{x_0})}+2|(\Nx P(u\otimes u),\varphi_{R,x_0}u)|.
\end{equation}
Here the constant $C$ is independent of $R$. To estimate the term containing pressure, we use Lemma \ref{Lem0.2} with $q=3$ and $p=3/2$. Then, due to
\eqref{0.15} together with the H\"older and Young inequalities,
\begin{multline}\label{0.huuge}
|(\Nx P(u\otimes u),\varphi_{R,x_0}u)|\le C\int_{\R^2}\theta_{R,x_0}(x)\|u\otimes u\|_{L^{3/2}(B^R_{x})}\,dx\cdot\|u\|_{L^3(B^{2R}_{x_0})}\le\\\le
C\int_{\R^2}\theta_{R,x_0}(x)\|u\|^2_{L^3(B^{R}_{x})}\,dx\cdot\|u\|_{L^3(B^{2R}_{x_0})}\le \\\le C\(\int_{\R^2}\theta_{R,x_0}(x)\,dx\)^{1/3}\(\int_{\R^2}\theta_{R,x_0}(x)\|u\|^3_{L^3(B^{R}_{x})}\,dx\)^{2/3}\cdot\|u\|_{L^3(B^{2R}_{x_0})}\le\\\le CR^{-1/3}\(\int_{\R^2}\theta_{R,x_0}(x)\|u\|^3_{L^3(B^{R}_{x})}\,dx\)^{2/3}\cdot\|u\|_{L^3(B^{2R}_{x_0})}\le\\\le C\int_{\R^2}\theta_{R,x_0}(x)\|u\|^3_{L^3(B^{R}_{x})}\,dx+
CR^{-1}\|u\|_{L^3(B^{2R}_{x_0})}^3,
\end{multline}
where all constants are independent of $R\gg1$. Thus, \eqref{Z.13} now reads
\begin{multline}\label{Z.14}
\frac d{dt}\|u(t)\|_{L^2_{\varphi_{R,x_0}}}^2+\alpha\|u\|^2_{L^2_{\varphi_{R,x_0}}}\le C\|g\|_{L^2_{\varphi_{R,x_0}}}^2+\\+
 CR^{-1}\|u\|^3_{L^3(B^{2R}_{x_0})}+C\int_{\R^2}\theta_{R,x_0}(x)\|u\|^3_{L^3(B^{R}_{x})}\,dx.
\end{multline}
We  introduce
\begin{equation}\label{13}
Z_{R,y_0}(u):=\int_{x_0\in\R^2}\theta_{R,y_0}(x_0)\|u\|^2_{L^2_{\varphi_{R,x_0}}}\,dx_0.
\end{equation}
Then, using \eqref{1.eqtheta}, we have
\begin{equation}\label{14}
C_2\int_{y\in\R^2}\theta_{R,y_0}(y)\|u\|^2_{L^2(B^R_y)}\,dy\le Z_{R,y_0}(u)\le C_1\int_{y\in\R^2}\theta_{R,y_0}(y)\|u\|^2_{L^2(B^R_y)}\,dy,
\end{equation}
where $C_i$ are independent of $R$. Multiplying now equation \eqref{Z.14} on $\theta_{R,y_0}(x_0)$, integrating over $x_0\in\R^2$ and using \eqref{1.R-thetakey}, we see that, for sufficiently large $R$,
\begin{equation}\label{12}
\frac d{dt}Z_{R,x_0}(u(t))+2\beta Z_{R,x_0}(u(t))\le C Z_{R,x_0}(g)+
 CR^{-1} \int_{x\in\R^2}\theta_{R,x_0}(x)\|u\|_{L^3(B^R_x)}^3\,dx,
\end{equation}
where the positive constants $C$ and $\beta$ are independent of $R$.
 \par
  Thus, we only need to estimate the integral in the RHS of \thetag{12} which, however, a bit more delicate than in \cite{Zel-JMFM} since now we do not have the control of the $H^1$-norm. We use the interpolation inequality proved in Appendix
\begin{equation}\label{A.1-copy}
\|u\|_{L^3(B_{R})}^3\le C\(\frac1R\|u\|^3_{L^2(B_{2R})}+\|u\|^{5/2}_{L^2(B_{2R})}\|\omega\|_{L^\infty(B_{2R})}^{1/2}\).
\end{equation}
which holds for every $R>0$ and every $u\in\Cal H_b$. Using this inequality, we have
\begin{multline}
 \int_{x\in\R^2}\theta_{R,x_0}(x)\|u\|_{L^3(B^R_x)}^3\,dx\le\\\le C\int_{x\in\R^2}\theta_{R,x_0}(x)\(\frac1R\|u\|^3_{L^2(B^{2R}_x)}+\|u\|^{5/2}_{L^2(B^{2R}_x)}\|\omega\|_{L^\infty(B^{2R}_x)}^{1/2}\)\,dx
\le\\\le C\int_{x\in\R^2}\theta_{R,x_0}(x)\|u\|^2_{L^2(B^{2R}_x)}\(\frac1R\|u\|_{L^2(B^{2R}_x)}+
\|u\|^{1/2}_{L^2(B^{2R}_x)}\|\omega\|_{L^\infty(B^{2R}_x)}^{1/2}\)\,dx\le\\\le C\(R^{-1}\|u\|_{L^2_{b,2R}}+\|u\|^{1/2}_{L^2_{b,2R}}\|\omega\|_{L^\infty}^{1/2}\)
\int_{x\in\R^2}\theta_{R,x_0}(x)\|u\|_{L^2(B^{2R}_x)}^2\,dx\le\\\le CR^{1/2}(\|\omega\|_{L^\infty}+R^{-1}\|u\|_{L^2_{b,R}})Z_{R,x_0}(u).
\end{multline}

Inserting this estimate into \thetag{12}, we finally get
\begin{multline}\label{21}
\frac d{dt} Z_{R,x_0}(u(t))+\beta Z_{R,x_0}(u(t))+\\+\(\beta-CR^{-1/2}(\|\omega(t)\|_{L^\infty}+R^{-1}\|u(t)\|_{L^2_{b,R}})\)Z_{R,x_0}(u(t))\le CZ_{R,x_0}(g).
\end{multline}
That is the complete analogue of estimate (5.18) of \cite{Zel-JMFM},
 so arguing exactly as in the proof of estimate (5.23) there, we end up with the estimate
\begin{equation}\label{22}
\|u(t)\|_{L^2_b}\le C(\|u_0\|_{\Cal H_b}+\|g\|_{\Cal H_b}+1)^3.
\end{equation}
For the convenience of the reader, we give below a schematic derivation of \eqref{22} from the key estimate \eqref{21}. The details can be found in \cite{Zel-JMFM}. Indeed, under the additional assumption that the constant $R=R(u_0,g)$ satisfies
\begin{equation}\label{4.loop}
KR^{-1/2}(R^{-1}\|u\|_{L^2_{b,R}}+\|\omega\|_{L^\infty})\le \beta, \ \ t\ge0,
\end{equation}
the Gronwall estimate applied to \eqref{21} gives
\begin{equation}\label{4.Zbound}
Z_{R,x_0}(t)\le Z_{R,x_0}(u_0)e^{-\beta t}+CZ_{R,x_0}(g)\le CR(\|u_0\|_{L^2_b}^2+\|g\|_{L^2_b}^2)
\end{equation}
and, therefore, taking into the account \eqref{9} and \eqref{1.ul-wR}, we have the desired control
\begin{equation}\label{4.bound}
R^{-1}\|u(t)\|_{L^2_{b,R}}+\|w(t)\|_{L^\infty}\le C_1(\|u_0\|_{\Cal H_b}+\|g\|_{\Cal H_b}+1).
\end{equation}
Thus, to finish the proof of \eqref{22}, we only need to fix the parameter $R$ in such way that our extra assumption \eqref{4.loop} is satisfied. Inserting the obtained estimate \eqref{4.bound} into the left-hand side of \eqref{4.loop}, we see that it will be formally satisfied if
\begin{equation}\label{R}
R^{-1/2}=\frac\beta{KC_1}(\|u_0\|_{L^2_b}+\|\rot u_0\|_{L^\infty}+1+\|g\|_{L^2_b}+\|\rot g\|_{L^\infty})^{-1}
\end{equation}
and this estimate together with \eqref{4.bound} gives  the desired estimate \thetag{22}. Of course, the above arguments are formal but they can be made rigorous exactly as in \cite{Zel-JMFM}.

Remind that estimate is still not dissipative in time. In order to obtain its dissipative analogue, we just need to take $R=R(t)$ depending on time and argue exactly as in Section 6 of \cite{Zel-JMFM} (see also \cite{PZ12} for the analogous estimate in the case of the Cahn-Hilliard equation in $\R^3$). Indeed, as shown there, if we replace \eqref{R} by
\begin{equation}\label{RR}
R(t)^{-1/2}=\frac\beta{KC_1}\(\|u_0\|_{\Cal H_b}e^{-\gamma t}+1+\|g\|_{\Cal H_b}\)^{-1},
\end{equation}
where $\gamma>0$ is small enough then, arguing exactly as above we obtain that
\begin{equation}\label{L2b-dis}
\|u(t)\|_{L^2_b}\le\|u\|_{L^2_{b,R}}\le C\|u_0\|_{\Cal H_b}^3e^{-\gamma t}+ C(1+\|g\|_{\Cal H_b})^3.
\end{equation}
Thus, the desired estimate \eqref{8} for the $L^2_b$-norm of the velocity field is obtained. Since the control of the $L^\infty$-norm of the vorticity has been already obtained in \eqref{9}, the theorem is proved.
\end{proof}
\begin{remark}\label{Rem3.NS} Remind that the analogue of the dissipative estimate \eqref{8} for the case of damped Navier-Stokes equation
\begin{equation}\label{NSa}
\Dt u+(u,\Nx)u+\alpha u+\Nx p=\nu\Dx u+g,\ \ \divv u=0
\end{equation}
has been previously obtained in \cite{Zel-JMFM}. However, the proof given there used essentially the viscous term $\nu\Dx u$ and the obtained estimate was not uniform with respect to $\nu\to0$. In contrast to this, based on the new version of the interpolation inequality, see \eqref{A.1-copy}, we have checked that the above estimate holds for the limit case $\nu=0$. Moreover, as not difficult to see, the dissipative estimate \eqref{8} is now {\it uniform} with respect to $\nu\to0$.
\end{remark}
\begin{remark}\label{Rem3.NS0} The method described above works in the case of classical Euler equations (which corresponds to $\alpha=0$) as well. However, in this case we cannot expect any dissipative estimates and the analogue of \eqref{8} will be growing in time:
\begin{equation}\label{Egrow}
\|u(t)\|_{\Cal H_b}\le C(1+\|g\|_{\Cal H_b}+\|u_0\|_{\Cal H_b})^3(t+1)^5.
\end{equation}
The proof of this estimate repeats word by word the one given in \cite{Zel-JMFM} for the case of damped Navier-Stokes equations (if we use the new version of the key interpolation inequality). In addition, the following estimate stated in \cite{Zel-JMFM} remains true:
\begin{equation}\label{R-dis}
\frac1{(t+1)^4}\|u(t)\|_{L^2_{b,(t+1)^4}}\le C(t+1),
\end{equation}
where the constant $C$ depends on $u_0$ and $g$, but is independent of $t$. As elementary examples with $g=const$, $u=t g$ show, in contrast to \eqref{Egrow},  estimate \eqref{R-dis} on the mean value of the energy with respect to the expanding balls of radii $R(t)=(t+1)^4$ is sharp.

 Moreover, in the important particular case $g=0$ this estimate can be essentially improved arguing exactly as in \cite{Zel-JMFM}:
\begin{equation}\label{R_0-dis}
\frac1{(t+1)^2}\|u(t)\|_{L^2_{b,(t+1)^2}}\le C,
\end{equation}
so the $L^2_b$-norm of the velocity field in this case can grow at most as a quadratic polynomial in time. The usage of the following $L^\infty$ analogue of the interpolation inequality \eqref{A.1-copy}:
\begin{equation}\label{inf-ineq}
\|u\|_{L^\infty(B^R_0)}\le C\(\|\rott u\|_{L^\infty(B^{2R}_0)}^{1/2}\|u\|_{L^2(B^{2R}_0)}^{1/2}+\frac1R\|u\|_{L^2(B^{2R}_0)}\)
\end{equation}
which can be proved analogously to \eqref{A.1-copy} allows us to improve essentially the inequality \eqref{Egrow}. Indeed,
applying \eqref{inf-ineq} for the vector field $u(t)$, fixing $R=(t+1)^4$ and using that
\begin{equation}
\|\rott u(t)\|_{L^\infty}\le C(\|\rott u_0\|_{L^\infty}+t\|\rott g\|_{L^\infty})
\end{equation}
(which is the analogue of \eqref{9} for the case of $\alpha=0$) together with estimate \eqref{R-dis}, we see that
\begin{equation}\label{R-better}
\|u(t)\|_{L^2_b}\le\|u(t)\|_{L^\infty}\le C(t+1)^3,
\end{equation}
where $C$ depends on $g$ and $u_0$ but is independent of $t$. Finally, in the particular case where $g=0$, taking $R=(t+1)^2$ and using estimate \eqref{R_0-dis}, we see that
\begin{equation}\label{R_0-better}
\|u(t)\|_{L^2_b}\le C(t+1),
\end{equation}
where $C$ depends on $u_0$, but is independent of $t$. Actually, we do not know whether or not the $\|u(t)\|_{L^2_b}$ norm can grow as $t\to\infty$. However, it has been recently established in \cite{Gal-Sli} that in the case of damped Navier-Stokes equation in an infinite cylinder (with the periodicity assumption with respect to one variable, say, $x_1$), the corresponding solution remains bounded as $t\to\infty$. We also mention that estimate \eqref{R_0-better} has been  recently obtained in \cite{Gal14} based on a slightly different representation of  the non-linearity and   pressure term in Navier-Stokes equation which allows  to avoid the usage of rather delicate interpolation inequality \eqref{inf-ineq}.
\end{remark}

\section{Uniqueness and enstrophy equality}\label{s4}
The aim of this section is to adapt the Yudovich proof of uniqueness for the Euler equations (see \cite{Yu1}) to the case of uniformly local spaces. The key technical thing for this proof is the following lemma.
\begin{lemma}\label{Lem4.Yud} Let the vector field $u\in\Cal H_b$. Then, the following estimate holds
\begin{equation}\label{7}
\|u\|_{W^{1,p}_b}\le Cp\|u\|_{\Cal H_b},
\end{equation}
where the constant $C$ is independent of $p>2$ and $u\in\Cal H_b$.
\end{lemma}
\begin{proof} Estimate \thetag{7} follows from the following analogous estimate in the case of bounded domains established by Yudovich.
\begin{proposition} Let $\Omega\subset\R^2$ be a smooth bounded domain. Then, for any vector field $v\in [W^{1,p}(\Omega)]^2$ such that $v.n\big|_{\partial\Omega}=0$ and any $1<p<\infty$, the following estimate holds:
\begin{equation}\label{4.Jureg}
\|v\|_{W^{1,p}(\Omega)}\le C(p+\frac1{p-1})\(\|\divv v\|_{L^p(\Omega)}+\|\rott v\|_{L^p(\Omega)}\),
\end{equation}
where the constant $C$ is independent of $p$ and $v$.
\end{proposition}
Indeed, due to Leray-Helmholtz decomposition, the vector field $v$ can be expressed in terms of inverse Laplacians as follows:
\begin{equation}
v=-\Nx(-\Dx)_N^{-1}\divv v-\Nx^{\perp}(-\Dx)^{-1}_D\rott v,
\end{equation}
where $(-\Dx)_D^{-1}$ and $(-\Dx)_N^{-1}$ are inverse Laplacians with Dirichlet and Neumann boundary conditions respectively and $\Nx^{\perp}:=(-\partial_{x_2},\partial_{x_1})$ is the orthogonal complement to the gradient. Thus, to verify \eqref{4.Jureg}, it is enough to know that
\begin{equation}\label{4.max}
\|Av\|_{W^{2,p}(\Omega)}\le C(p+\frac1{p-1})\|v\|_{L^p(\Omega)},
\end{equation}
for the case $A=(-\Dx)_D^{-1}$ and $A=(-\Dx)_N^{-1}$ (of course, for the case of Neumann boundary conditions we need the extra zero mean assumption). The proof of estimate \eqref{4.max} can be found in \cite{Yu2}. Thus, estimate \eqref{4.Jureg} is verified and we may return to the proof of the desired estimate \eqref{7}. Let $\varphi_{x_0}$ be the smooth cut-off function which equals one identically if $x\in B^1_{x_0}$ and zero outside of the ball $B^2_{x_0}$. Then applying \eqref{4.Jureg} to the function $v:=\varphi_{x_0}u$ and $\Omega=B^2_{x_0}$, we get
\begin{equation}
\|u\|_{W^{1,p}(B^1_{x_0})}\le C p(\|\rott u\|_{L^p(B^2_{x_0})}+\|u\|_{L^\infty(B^2_{x_0})})\le Cp(\|\rott u\|_{L^\infty(B^3_{x_0})}+\|u\|_{L^2(B^3_{x_0})}),
\end{equation}
where we have used the obvious estimate $\|u\|_{L^\infty(B^2_{x_0})}\le C(\|\rott u\|_{L^{\infty}(B^3_{x_0})}+\|u\|_{L^2(B^3_{x_0})})$ and the fact that $p>2$ is separated from the singularity at $p=1$. Taking the supremum over $x_0\in\R^2$, we obtain the desired estimate \eqref{7} and finish the proof of the lemma.
\end{proof}

The main result of the section is the following theorem.
\begin{theorem}\label{Th4.uni} Let $u_1(t)$ and $u_2(t)$ be two weak solutions of the damped Euler equation \eqref{1}. Then, the following estimate holds:
\begin{equation}\label{4.unilip}
\|u_1(t)-u_2(t)\|_{L^2_b}^2\le K e\(\frac{\|u_1(0)-u_2(0)\|_{L^2_b}^2}{K}\)^{e^{-Lt}},
\end{equation}
where the positive constants $K$ and $L$ depend on the $\Cal H_b$-norms of $u_1(0)$ and $u_2(0)$, but are independent of $t$. In particular, the weak solution of the damped Euler solution is unique.
\end{theorem}
\begin{proof}
Let $u_1$ and $u_2$ be two solutions of \thetag{1} and $w=u_1-u_2$. Then, this function solves
\begin{equation}\label{23}
\Dt w+(w,\Nx )u_1+(u_2,\Nx)w+\alpha w+\Nx p=0.
\end{equation}
Multiplying \thetag{23} by $w\varphi_{R,x_0}$, where $\varphi_{R,x_0}$ is the same as in \eqref{11} and $R>1$ and $x_0\in\R^2$ are arbitrary, after the straightforward calculations, we have
\begin{multline}
\frac12\frac d{dt}\|w\|^2_{L^2_{\varphi_{R,x_0}}}+\alpha \|w\|^2_{L^2_{\varphi_{R,x_0}}}\le C(|\Nx u_1|+|\Nx u_2|,w^2\varphi_{R,x_0})+\\+
CR^{-1}(|u_2|,w^2)_{L^2(B^{2R}_{x_0})}+|(\Nx P(u_1\otimes u_1-u_2\otimes u_2),w\varphi_{R,x_0})|.
\end{multline}
To estimate the term with pressure, we use \eqref{0.15} with $p=2$ which gives
\begin{multline}
|(\Nx P(u_1\otimes u_1-u_2\otimes u_2),w\varphi_{R,x_0})|\le\\\le C\|w\|_{L^2(B^{2R}_{x_0})}\int_{y\in\R^2}\theta_{R,x_0}(y)\|u_1\otimes u_1-u_2\otimes u_2\|_{L^2(B^R_{y})}\,dy.
\end{multline}
Using now that $\|u_i\|_{L^\infty}\le C$, $i=1,2$, where the constant $C$ depends on the $\Cal H_b$-norms of the initial data (thanks to the dissipative estimate \eqref{8} and the obvious embedding $\Cal H_b\subset L^\infty$), together with the Cauchy-Schwartz inequality and the straightforward inequality
$$
\|u\|_{L^2(B^{2R}_{x_0})}^2\le C_R\int_{y\in\R^2}\theta_{R,x_0}(y)\|u\|^2_{L^2(B^{R}_{y})}\,dy,
$$
we end up with
\begin{multline}\label{24}
\frac12\frac d{dt}\|w\|^2_{L^2_{\varphi_{R,x_0}}}+\alpha \|w\|^2_{L^2_{\varphi_{R,x_0}}}\le C_R\int_{y\in\R^2}\theta_{R,x_0}(y)\|w\|^2_{L^2(B^{R}_{y})}\,dy+\\+ C(|\Nx u_1|+|\Nx u_2|,w^2\varphi_{R,x_0}).
\end{multline}
Thus, we only need to estimate the most complicated last term in the RHS of \thetag{24}. To this end, we will essentially use \thetag{7}
and the fact that $\|u_i\|_{\Cal H_b}\le C$. Then,  due to the interpolation inequality
$$
\|w\|_{L^{2p/(p-1)}}\le C\|w\|_{L^\infty}^\theta\|w\|_{L^2}^{1-\theta},\ \ \theta=\frac1p
$$
which holds for any $p>2$, we end up with
\begin{multline}
|(|\Nx u_1|+|\Nx u_2|,w^2\varphi_{R,x_0})|\le C_R(\|u_1\|_{W^{1,p}_b}+\|u_2\|_{W^{1,p}_b})\|w\|^2_{L^{2p/(p-1)}(B^{2R}_{x_0})}\le\\\le C p\|w\|_{L^\infty}^{2/p}\|w\|_{L^2(B^{2R}_{x_0})}^{2(p-1)/p}\le C p\|w\|^{2(p-1)/p}_{L^2(B^{2R}_{x_0})}
\end{multline}
Let us take here $p=\ln\(\frac{K}{\|w\|^2_{L^2(B^{2R}_{x_0})}}\)$, where $K$ is large enough to guarantee that $p>2$. Such $K=K(\|u_1(0)\|_{\Cal H_b},\|u_2(0)\|_{\Cal H_b})$ exists since $u_1$ and $u_2$ are globally bounded in the $L^2_b$-norm and, consequently,
$$
\|w\|_{L^2(B^{2R}_{x_0})}\le \|u_1\|_{L^2_{b,2R}}+\|u_2\|_{L^2_{b,R}}\le Q_R(\|u_1(0)\|_{\Cal H_b}+\|u_2(0)\|_{\Cal H_b})
$$
for some monotone increasing function $Q_R$.
Then, we get

$$
|(|\Nx u_1|+|\Nx u_2|,w^2\varphi_{R,x_0})|\le C\|w\|^2_{L^2(B^{2R}_{x_0})}\ln\(\frac{K}{\|w\|^2_{L^2(B^{2R}_{x_0})}}\)
$$
and \thetag{24} reads
\begin{equation}\label{25}
\frac d{dt}\|w\|^2_{L^2_{\varphi_{R,x_0}}}\le C_R\int_{y\in\R^2}\theta_{R,x_0}(y)\|w\|^2_{L^2(B^{R}_{y})}\,dy  +C\|w\|^2_{L^2(B^{2R}_{x_0})}\ln\(\frac{K}{\|w\|^2_{L^2(B^{2R}_{x_0})}}\).
\end{equation}
We now use that the function $z\to z\log \frac K{z}$ is concave. Then, due to Jensen inequality
\begin{multline}
\int_{x_0\in\R^2}\theta_{R,y}(x_0)\|w\|^2_{L^2(B^{2R}_{x_0})}\ln\(\frac{K}{\|w\|^2_{L^2(B^{2R}_{x_0})}}\)\,dx_0\le\\\le
 \int_{x_0\in\R^2}\theta_{R,y}(x_0)\|w\|^2_{L^2(B^{2R}_{x_0})}\,dx_0\ln\(\frac {K\int_{x_0\in\R^2}\theta_{R,y}(x_0)\,dx_0}{\int_{x_0\in\R^2}\theta_{R,y}(x_0)\|w\|^2_{L^2(B^{2R}_{x_0})}\,dx_0}\).
\end{multline}
Using also that the function $z\to z\ln\frac Kz$ is monotone increasing if $z\le Ke^{-1}$, together with \eqref{14} and \eqref{1.eqtheta}, we get
$$
\int_{x_0\in\R^2}\theta_{R,y}(x_0)\|w\|^2_{L^2(B^{2R}_{x_0})}\ln\(\frac{K}{\|w\|^2_{L^2(B^{2R}_{x_0})}}\)\,dx_0\le C Z_{R,y}(w)\ln \frac {K_1}{Z_{R,y}(w)}
$$
for some new constant $K_1>0$.
Multiplying now \thetag{25} by $\theta_{R,y}(x_0)$, $y\in\R^2$, integrating over $x_0\in\R^2$ and \eqref{1.thetakey},
we finally arrive at
\begin{equation}\label{26}
\frac d{dt} Z_{R,y}(w)\le L Z_{R,y}(w)+C Z_{R,y}(w)\ln\frac{K}{Z_{R,y}(w)}
\end{equation}
for some new constants $L$ and $K$ depending on the $\Cal H_b$-norms of the initial data. Integrating this inequality, we have
\begin{equation}\label{4.loclip}
Z_{R,y}(w(t))\le K\(\frac {Z_{R,y}(w(0))}K\)^{e^{-Lt}}e^{1-e^{-Lt}}.
\end{equation}
Fixing, say, $R=1$ in this estimate and taking the supremum over $y\in\R^2$, we end up with the desired estimate \eqref{4.unilip} and finish the proof of the theorem.
\end{proof}
We conclude this section by reminding the so-called weighted enstrophy equality which will be used in the next section in order to verify the convergence to the attractor in a strong topology. Indeed, multiplying formally equation \eqref{3.vort} by $\omega\phi_{\eb,x_0}$ where $\phi_{\eb,x_0}(x):=e^{-\eb|x-x_0|}$, integrating over $x\in\R^2$ and using that $\divv u=0$, we arrive at
\begin{equation}\label{4.enst}
\frac12\frac d{dt}\|\omega\|^2_{L^2_{\phi_{\eb,x_0}}}+\alpha\|\omega\|^2_{L^2_{\phi_{\eb,x_0}}}-(u.\Nx\phi_{\eb,x_0},|\omega|^2)=(\rott g,\omega\phi_{\eb,x_0}).
\end{equation}
However, these arguments require justification since for weak solutions $\omega\in L^\infty([0,T]\times\R^2)$ only and the term $((u,\Nx)\omega,\omega\phi_{\eb,x_0})$ is not rigorously defined. We overcome this difficulty using the mollification operators and arguing as in \cite{DiL}.
\begin{theorem} Let $u$ be a weak solution of the damped Euler problem \eqref{1} and let $\eb>0$ and $x_0\in\R^2$ be arbitrary. Then, the function $t\to\|\omega(t)\|^2_{L^2_{\phi_{\eb,x_0}}}$ is absolutely continuous and the equality \eqref{4.enst} holds for almost all $t$.
\end{theorem}
\begin{proof} Indeed, let $S_\mu v:=\rho_\mu*v$ where $\rho_\mu(x)=\mu^{-2}\phi(x\mu^{-1})$, $\mu>0$ and $\rho$ is a standard mollification kernel. Then, applying $S_\mu$ to both sides of equation \eqref{3.vort} and denoting $\omega_\mu:=S_\mu\omega$, we have
\begin{equation}\label{4.sm}
\Dt\omega_\mu+\alpha\omega_\mu+(u,\Nx)\omega_\mu=\rott g_\mu+R_\mu,
\end{equation}
where $R_\mu:=((u,\Nx)\omega))*\rho_\mu-(u,\Nx)(\omega*\rho_\mu)$. Using the fact that $u(t)\in W^{1,p}_b(\R^2)$ for all $p<\infty$ and arguing exactly as in \cite{DiL}, we see that $R_\mu$ is uniformly with respect to $\mu\to0$ bounded in $L^\infty([0,T],L^p_b(\R^2))$ and
\begin{equation}\label{4.conv}
R_\mu\to0\ \ \text{ in } L^1([0,T], L^p_{loc}(\R^2)),
\end{equation}
see \cite{DiL}, Lemma II.1, page 516. Multiplying  \eqref{4.sm} by $\omega_\mu\phi_{\eb,x_0}$ (which is now allowed since $\omega_\mu$ is smooth in $x$) after the integration over $x$ and $t$, we get that for every $t\ge s\ge0$
\begin{multline}
\frac12\(\|\omega_\mu(t)\|_{L^2_{\phi_{\eb,x_0}}}^2-\|\omega_\mu(s)\|_{L^2_{\phi_{\eb,x_0}}}^2\)=
\int^t_s\((\rott g_\mu,\omega_\mu(\tau)\phi_{\eb,x_0})+\right.\\\left.+(R_\mu(\tau),\omega_\mu(\tau)\phi_{\eb,x_0})+
(u(\tau).\Nx\phi_{\eb,x_0},|\omega_\mu(\tau)|^2)-\alpha\|\omega_\mu(\tau)\|^2_{L^2_{\phi_{\eb,x_0}}}\)\,d\tau.
\end{multline}
Passing to the limit $\mu\to0$ in this equality and using \eqref{4.conv} together with the standard convergence properties of the mollification operators as well as that
\begin{equation}
\omega\in L^\infty(R_+\times\R^2)\cap C_w([0,T],L^2_{\phi_{\eb,x_0}}(\R^2)),
\end{equation}
we end up with the integral equality
\begin{multline}
\frac12\(\|\omega(t)\|_{L^2_{\phi_{\eb,x_0}}}^2-\|\omega(s)\|_{L^2_{\phi_{\eb,x_0}}}^2\)=
\int^t_s\((\rott g,\omega(\tau)\phi_{\eb,x_0})+\right.\\\left.+
(u(\tau).\Nx\phi_{\eb,x_0},|\omega(\tau)|^2)-\alpha\|\omega(\tau)\|^2_{L^2_{\phi_{\eb,x_0}}}\)\,d\tau
\end{multline}
which is equivalent to \eqref{4.enst} and finish the proof of the theorem.
\end{proof}

\section{The attractor}\label{s5}
The aim of this section is to verify the existence of the attractor for the damped Euler equation in the uniformly local spaces.
We first remain that according to Theorems \ref{Th3.main} and \ref{Th4.uni}, equation \eqref{1} generates a solution semigroup $S(t)$ in the phase space $\Cal H_b$:
\begin{equation}\label{5.sem}
S(t):\Cal H_b\to\Cal H_b,\ \ S(t)u_0:=u(t),\ \ t\ge0,
\end{equation}
where $u(t)$ is a unique solution of \eqref{1} with the initial data $u_0\in\Cal H_b$. Moreover, according to estimate \eqref{8} it is dissipative in the space $\Cal H_b$, i.e., the following estimate holds:
\begin{equation}\label{5.dis}
\|S(t)u_0\|_{\Cal H_b}\le Q(\|u_0\|_{\Cal H_b})e^{-\beta t}+Q(\|g\|_{\Cal H_b})
\end{equation}
for some positive $\beta$ and monotone increasing $Q$ which are independent of $t$ and $u_0$ and, according to estimates \eqref{4.unilip} and \eqref{4.loclip}, the maps $S(t)$ are locally H\"older continuous in the space $L^2_b(\R^2)$ and as well as in the space  $L^2_{\theta_{R,x_0}}(\R^2)$.
\par
As usual, in the case of unbounded domains and infinite energy solutions, see \cite{MZ08} for more details, we cannot expect the existence of a global attractor in the uniform topology of $\Cal H_b$, but only in the local topology of
\begin{equation}
\Cal H_{loc}:=\{u\in [L^2_{loc}(\R^2)]^2,\ \ \divv u=0,\ \ \omega\in L^\infty_{loc}(\R^2)\}.
\end{equation}
However, we do not know whether or not the above defined semigroup $S(t)$ is asymptotically compact in the strong topology of $\Cal H_b$, so we have to use the weak star topology in $\Cal H_{loc}$ (which we will further denote by $\Cal H_{loc}^{w^*}$) in order to define the convergence to the global attractor. We recall that a sequence $u_n\in\Cal H_{loc}$ converges weakly star in $\Cal H_{loc}$ to some function $u\in\Cal H_{loc}$ iff for any ball $B^R_{x_0}$ the restrictions $u_n\big|_{B^R_{x_0}}$ converge weakly to $u\big|_{B^R_{x_0}}$ in $L^2(B^R_{x_0})$ and the restrictions $\omega_n\big|_{B^R_{x_0}}$ converge weakly star to $\omega\big|_{B^R_{x_0}}$ in the space $L^\infty(B^R_{x_0})$. Remind also that any closed ball in $\Cal H_b$ is {\it metrizable} and is {\it compact} in the topology of $\Cal H_{loc}^{w^*}$, see \cite{RR}. Thus, we will use the following version of a global attractor.

\begin{definition}\label{Def5.attr} Let $S(t):\Cal H_b\to\Cal H_b$ be a semigroup. Then, a set  $\Cal A\subset\Cal H_b$ is a weak locally compact attractor for this semigroup iff:
\par
1) The set $\Cal A$ is bounded and closed in $\Cal H_b$ and is compact in the topology of $\Cal H_{loc}^{w^*}$;
\par
2) It is strictly invariant: $S(t)\Cal A=\Cal A$ for all $t\ge0$;
\par
3) It attracts bounded (in the topology of $\Cal H_b$) sets in the topology of $\Cal H_{loc}^{w^*}$, i.e., for every bounded set $B\subset \Cal H_b$ and every neighbourhood $\Cal O(\Cal A)$ of the attractor $\Cal A$ in the topology of $\Cal H_{loc}^{w^*}$, there exists $T=T(B,\Cal O)$ such that
\begin{equation}
S(t)B\subset\Cal O(\Cal A)
\end{equation}
for all $t\ge T$.
\end{definition}
The main result of the section is the following theorem.
\begin{theorem}\label{Th5.attr} Let the above assumptions hold. Then the solution semigroup $S(t):\Cal H_b\to\Cal H_b$ associated with the damped Euler equation \eqref{1} possesses a weak locally compact attractor $\Cal A$ in $\Cal H_b$ which is generated by all bounded solutions of the equation:
\begin{equation}\label{5.str}
\Cal A=\Cal K\big|_{t=0},
\end{equation}
where $\Cal K\subset L^\infty(\R,\Cal H_b)$ is the set of all weak solutions $u(t)$ of equation \eqref{1} which are defined for all $t\in\R$ and are bounded in $\Cal H_b$.
\end{theorem}
\begin{proof} Indeed, according to the dissipative estimate \eqref{5.dis}, the ball
\begin{equation}
\Cal B_R:=\{u\in\Cal H_b,\ \|u\|_{\Cal H_b}\le R\}
\end{equation}
is an absorbing ball for the semigroup $S(t)$ if $R$ is large enough. This ball is metrizable and compact in the weak star topology of $\Cal H_{loc}^{w^*}$. Thus, the considered semigroup possesses an absorbing ball $\Cal B_R$ which is bounded in $\Cal H_b$ and is compact in $\Cal H_{loc}^{w^*}$. Moreover, using the fact that the semigroup is H\"older continuous on $\Cal B_R$ (due to estimate \eqref{4.loclip}) together with the compactness of the embedding $\Cal H_{loc}\subset L^2_{loc}$, it is straightforward to show that, for every fixed $t\ge0$, the operators $S(t)$ are continuous on $\Cal B_R$ in the topology of $\Cal H_{loc}^{w^*}$. Thus, all assumptions of the abstract attractor existence theorem (see e.g., \cite{BV89}) are satisfied and the existence of the attractor $\Cal A$ is proved. Formula \eqref{5.str} for the attractor's structure is also an immediate corollary of this theorem. So, Theorem \ref{Th5.attr} is proved.
\end{proof}
\begin{corollary}\label{Cor5.ws} Let the above assumptions hold. Then, for every $\eb>0$ and every $p<\infty$,
 the weak locally compact attractor $\Cal A$ is compact in $[W^{1-\eb,p}_{loc}(\R^2)]^2$ and attracts the images of bounded sets in $\Cal H_b$ in the strong topology of $W^{1-\eb,p}_{loc}(\R^2)$, i.e., for every bounded set $B\subset \Cal H_b$ and every $R>0$ and $x_0\in\R^2$
 \begin{equation}\label{5.ws}
 \lim_{t\to\infty}\operatorname{dist}_{W^{1-\eb,p}(B^R_{x_0})}\((S(t)B)\big|_{B^R_{x_0}},\Cal A\big|_{B^R_{x_0}}\)=0,
 \end{equation}
 where $\operatorname{dist}_V(X,Y)$ is a non-symmetric Hausdorff distance between sets $X$ and $Y$ of a metric space $V$.
\end{corollary}
Indeed, the convergence \eqref{5.ws} is an immediate corollary of the definition  of the attractor $\Cal A$ and the compactness of the embedding $\Cal H_{loc}\subset W^{1-\eb,p}_{loc}(\R^2)$.
\par
We conclude this section by establishing, analogously to \cite{CVZ11}, that we may take $\eb=0$ in \eqref{5.ws}.

\begin{theorem}\label{Prop5.strong} Let the above assumptions hold. Then the attractor $\Cal A$ of the solution semigroup associated with the damped Euler equation \eqref{1} is compact in $W^{1,p}_{loc}(\R^2)$ for any $p<\infty$ and attracts the images of bounded sets in $\Cal H_b$ in the topology of this space.
\end{theorem}
\begin{proof} Indeed, due to the interpolation, it is sufficient to verify the asymptotic compactness of $S(t)$ in $W^{1,2}_{loc}(\R^2)$ or, which is the same, the asymptotic compactness of the associated vorticity $\omega$ in $L^2_{loc}(\R^2)$. To verify it, following \cite{CVZ11}, we will use the so-called energy method, see also \cite{B,MRW}. Let $u_0^n\in\Cal B_R$ be a sequence of the initial data and $t_n\to\infty$ be a sequence of times. We need to verify that the sequence $S(t_n)u_0^n$ is precompact in $W^{1,2}_{loc}(\R^2)$.
\par
Let $u_n(t)$, $t\ge -t_n$, be the solutions of the following damped Euler problems:
\begin{equation}\label{5.eun}
\Dt u_n+(u_n,\Nx)u_n+\Nx p_n+\alpha u_n=g,\ \ u_n\big|_{t=-t_n}=u_0^n.
\end{equation}
and the associated vorticities $\omega_n=\rott u_n$ solve
\begin{equation}\label{5.vortn}
\Dt \omega_n+(u_n,\Nx)\omega_n+\alpha \omega_n=\rott g,\ \ \omega_n\big|_{t=0}=\rott u_0^n.
\end{equation}
To verify the desired asymptotic compactness (and to finish the proof of the theorem), we only need to verify that the sequence $\omega_n(0)$ is precompact in $L^2_{loc}(\R^2)$. We first note that, due to the dissipative estimate \eqref{8}, the sequence $u_n$ is uniformly bounded in $\Cal H_b$:
\begin{equation}\label{5.bound}
\|u_n\|_{L^\infty(\mathbb R,\Cal H_b)}+\|\Dt u_n\|_{L^\infty(\R,L^q_b(\R^2))},\ \ q<\infty,
\end{equation}
where the control over the norm of $\Dt u_n$ is obtained from equation \eqref{5.eun} analogously to Definition \ref{Def0.1} (to simplify the notations, we extend $u_n$ and $\Dt u_n$ by zero for $t\le t_n$). Thus, without loss of generality, we may assume that
\begin{equation}\label{5.w-conv}
u_n\to u\ \ \text{weakly star in }\ \ L^\infty_{loc}(\R,\Cal H_{loc}),\ \ \Dt u_n\to\Dt u\ \ \text{weakly star in}\ L^\infty_{loc}(\R,L^q_{loc}(\R^2)).
\end{equation}
Then, due to the compactness arguments,
\begin{equation}\label{5.s-conv}
u_n\to u \ \ \text{ strongly in } C_{loc}(\R\times\R^3)
\end{equation}
and, in particular,
\begin{equation}\label{5.s-vort}
\omega_n\to\omega \ \ \text{ weakly star in } L^\infty_{loc}(\R\times\R^2)\ \ \text{ and}\ \ \omega_n\to\omega \ \ \text{strongly in }\
C_{loc}(\R,W^{-1,2}_{loc}(\R^2))
\end{equation}
Passing to the limit $n\to\infty$ in a straightforward way in equations \eqref{5.eun}, we see that the limit function $u(t)$, $t\in\R$, solves the damped Euler equation \eqref{1} and, therefore, $u\in\Cal K$. Moreover, from \eqref{5.s-vort}, we conclude that
\begin{equation}\label{5.0s-weak}
\omega_n(0)\to\omega(0)\ \ \text{weakly in }\ L^2_{loc}(\R^2)
\end{equation}
and using that $\omega_n(0)$ is uniformly bounded in $L^2_b(\R^2)$ the last convergence implies that
\begin{equation}\label{5.0w-weak}
\omega_n(0)\to\omega(0)\ \ \text{weakly in }\ L^2_{\phi_{\eb,x_0}}(\R^2)
\end{equation}
for all $\eb>0$ and $x_0\in\R^2$.
\par
At the second step, we will show that the convergence in \eqref{5.0w-weak} is actually {\it strong} which will complete the proof of the theorem. To this end, it is enough to prove that
\begin{equation}\label{5.norm-conv}
\|\omega_n(0)\|_{L^2_{\phi_{\eb,x_0}}}\to\|\omega(0)\|_{L^2_{\phi_{\eb,x_0}}}
\end{equation}
for some $\eb>0$ and $x_0\in\R^2$. To this end, we will use the enstrophy equality \eqref{4.enst} for equations \eqref{5.vortn} which we rewrite in the following form:
\begin{multline}\label{5.hugen}
\|\omega_n(0)\|^2_{L^2_{\phi_{\eb,x_0}}}+\\+\int_{-t_n}^0e^{\alpha s}\int_{x\in\R^2}(\alpha-\phi_{\eb,x_0}(x)^{-1}u_n(x,s).\Nx\phi_{\eb,x_0}(x))\phi_{\eb,x_0}(x)|\omega_n(s,x)|^2\,dx\,ds=\\=
\|\omega_n(-t_n)\|^2_{L^2_{\phi_{\eb,x_0}}}e^{-\alpha t_n}+\int_{-t_n}^0e^{\alpha s}(\rott g,\omega_n(s)\phi_{\eb,x_0})\,ds.
\end{multline}
Remind that
\begin{equation}\label{5.exp}
|\Nx\phi_{\eb,x_0}(x)|\le C\eb\phi_{\eb,x_0}(x)
\end{equation}
and therefore we may fix $\eb>0$ being small enough that
\begin{equation}
\alpha-\phi_{\eb,x_0}(x)^{-1}u_n(x,s).\Nx\phi_{\eb,x_0}(x)\ge0
\end{equation}
for all $(s,x)\in\R_-\times\R^2$. Then, using the classical result on the weak lower semicontinuity of convex functionals, see e.g., \cite{Io}, together with the strong convergence \eqref{5.s-conv} and weak convergence \eqref{5.s-vort}, we conclude that
\begin{equation*}
\int_{s\in\R_-}\int_{x\in\R^2}F(s,x,u(s,x),\omega(s,x))\,ds\,dx\le
\liminf_{n\to\infty}\int_{-t_n}^0\int_{x\in\R^2}F(s,x,u_n(s,x),\omega_n(s,x))\,ds\,dx,
\end{equation*}
where
\begin{equation}
F(s,x,u,\omega):=e^{\alpha s}(\alpha-\phi_{\eb,x_0}(x)^{-1}u.\Nx\phi_{\eb,x_0}(x))\phi_{\eb,x_0}(x)|\omega|^2.
\end{equation}
Passing now to the limit $n\to\infty$ in \eqref{5.hugen}, we arrive at
\begin{multline}\label{5.hugelim}
\limsup_{n\to\infty}\|\omega_n(0)\|^2_{L^2_{\phi_{\eb,x_0}}}+\int_{-\infty}^0\int_{x\in\R^2}F(s,x,u(s,x),\omega(s,x))dx\,ds\le\\\le
\int_{-\infty}^0e^{\alpha s}(\rott g,\omega(s)\phi_{\eb,x_0})\,ds.
\end{multline}
On the other hand, according to the enstrophy equality for the limit functions $u$ and $\omega$,
\begin{equation}\label{5.hugeattr}
\|\omega(0)\|^2_{L^2_{\phi_{\eb,x_0}}}+\int_{-\infty}^0\int_{x\in\R^2}F(s,x,u(s,x),\omega(s,x))dx\,ds=
\int_{-\infty}^0e^{\alpha s}(\rott g,\omega(s)\phi_{\eb,x_0})\,ds.
\end{equation}
Thus,
\begin{equation}
\limsup_{n\to\infty}\|\omega_n(0)\|^2_{L^2_{\phi_{\eb,x_0}}}\le
\|\omega(0)\|^2_{L^2_{\phi_{\eb,x_0}}}\le\liminf_{n\to\infty}\|\omega_n(0)\|^2_{L^2_{\phi_{\eb,x_0}}}.
\end{equation}
Therefore, the convergence \eqref{5.norm-conv} is verified and the theorem is proved.
\end{proof}

\section{Appendix. The interpolation inequality}\label{sA}
The aim of this Appendix is to verify the following interpolation inequality.
\begin{lemma} Let  $u\in L^2_b(\R^2)$ be the divergent free vector field  such that $\omega:= \rott u\in L^\infty(\R^2)$. Then, the following inequality holds:
\begin{equation}\label{A.1}
\|u\|_{L^3(B^{R}_{x_0})}^3\le C\(\frac1R\|u\|^3_{L^2(B^{2R}_{x_0})}+\|u\|^{5/2}_{L^2(B^{2R}_{x_0})}\|\omega\|_{L^\infty(B^{2R}_{x_0})}^{1/2}\).
\end{equation}
where $R>0$, $x_0\in\R^2$ are arbitrary and the constant $C$ is independent of $R$, $x_0$, and $u$.
\end{lemma}
To verify this inequality, we use the following result proved in \cite{Zel-JMFM}.
\begin{proposition}\label{Lem4.key}
  Let the vector field $u\in [W^{1,2}_0(B^{2R}_{x_0})]^2$ be such that $\divv u,\rot u\in L^\infty(B^{2R}_{x_0})$. Then,
\begin{equation}\label{4.keyest}
\|u\|_{L^3(B^{2R}_{x_0})}\le C\|u\|_{L^2(B^{2R}_{x_0})}^{5/6}\(\|\rot u\|_{L^\infty(B^{2R}_{x_0})}+\|\divv u\|_{L^\infty(B^{2R}_{x_0})}\)^{1/6},
\end{equation}
where the constant $C$ is independent of $R$ and $x_0$. Moreover, for any $2<p<\infty$,
\begin{equation}\label{4.keyest1}
\|u\|_{L^\infty(B^{2R}_{x_0})}\le C\|u\|_{L^2(B^{2R}_{x_0})}^{\theta}\(\|\rot u\|_{L^p(B^{2R}_{x_0})}+\|\divv u\|_{L^p(B^{2R}_{x_0})}\)^{1-\theta},
\end{equation}
where $ \theta=\frac12-\frac1{2(p-1)}$, $C$ may depend on $p$, but is independent of $R$ and $x_0\in\R^2$.
\end{proposition}
\begin{proof}[Proof of the lemma]
We first note that \thetag{A.1} is homogeneous, so (scaling $x\to Rx$ if necessary) it is enough to prove it for $R=1$ and $x_0=0$ only. Let now $\varphi\in C^\infty_0(B^{4/3}_0)$ be the cut-off function such that $\varphi(x)\equiv 1$ if $x\in B^1_0$. Then applying inequality \eqref{4.keyest} with $R=2$  to the vector field $\varphi u$, we get
\begin{equation}\label{A.2}
\|u\|_{L^3(B^1_0)}^3 \le \|\varphi u\|_{L^3(B^{4/3}_0)}^3\le C\|u\|_{L^2(B^2_0)}^{5/2}\|\omega\|_{L^\infty(B^2_0)}^{1/2}+C\|u\|_{L^2(B^2_0)}^{5/2}\| u\|_{L^\infty(B^{4/3}_0)}^{1/2}.
\end{equation}
Applying now inequality \eqref{4.keyest1} with $p=4$ and $\theta=1/3$ to the vector field $\varphi_1 u$, where $\varphi_1\in C_0^\infty(B^{5/3}_0)$
is a new cut-off function which equals to one if $x\in B^{4/3}_0$, we get
\begin{equation}\label{A.3}
\|u\|_{L^\infty(B^{4/3}_0)}^{1/2}\le C\|u\|_{L^2(B^2_0)}^{1/6}\|\omega\|_{L^4(B^{2}_0)}^{1/3}+ \|u\|_{L^2(B^2_0)}^{1/6}\|u\|^{1/3}_{L^4(B^{5/3}_0)}.
\end{equation}
In the first term, we replace the $L^4$-norm of $\omega$ by its $L^\infty$-norm, insert the obtained result to \eqref{A.2} and use the Young inequality, this gives
\begin{multline}\label{A.4}
\|u\|^3_{L^3(B^1_0)}\le C\|u\|_{L^2(B^2_0)}^{5/2}\|\omega\|_{L^\infty(B^2_0)}^{1/2}+
C\|u\|^{8/3}_{L^2(B^2_0)}\|\omega\|^{1/3}_{L^\infty(B^2_0)}+\\+C\|u\|^{8/3}_{L^2(B^2_0)}\|u\|^{1/3}_{L^4(B^{5/3}_0)}\le  C\|u\|_{L^2(B^2_0)}^3+ C\|u\|_{L^2(B^2_0)}^{5/2}\|\omega\|_{L^\infty(B^2_0)}^{1/2}+\\+ C\|u\|^{8/3}_{L^2(B^2_0)}\|u\|^{1/3}_{L^4(B^{5/3}_0)}.
\end{multline}
Thus, we only need to estimate the $L^4$-norm in the RHS of \eqref{A.4}. To this end, we introduce one more cut-off function $\varphi_2\in C_0^\infty(B^2_0)$ such that $\varphi_2(x)=1$ for $x\in B^{5/3}_0$ and use the following Ladyzhenskaya type inequality for vector fields $v\in W^{1,2}_0(B^2_0)$:
\begin{equation}\label{A.5}
\|v\|^4_{L^4(B^2_0)}\le C\|v\|_{L^2(B^2_0)}^2\(\|\divv v\|^2_{L^2(B^2_0)}+\|\rott v\|^2_{L^2(B^2_0)}\).
\end{equation}
Applying this inequality to the vector field $\varphi_2 u$ and estimating again the $L^2$-norm of the vorticity by its $L^\infty$-norm, we have
\begin{equation}\label{A.6}
\|u\|^{1/3}_{L^4(B^{5/3}_0)}\le C\|u\|^{1/6}_{L^2(B^2_0)}\|\omega\|^{1/6}_{L^\infty(B^2_0)}+C\|u\|_{L^2(B^2_0)}^{1/3}.
\end{equation}
Inserting this estimate to the RHS of \eqref{A.4} and using Young inequality again, we derive the desired estimate \eqref{A.1}.
\end{proof}

\end{document}